\documentclass[smallextended,referee,envcountsect,]{svjour3}
\smartqed

\begin{document}

\title{New strategies for some pursuit-evasion differential games with Gronwall-type constraints}
\titlerunning{New strategies for some pursuit-evasion differential games}

\author{Jamilu Adamu$^{1,4}$. Mehdi Salimi$^{2,\clubsuit}$. Farouq Abba Wabi$^{3}$. Jewaidu Rilwan$^{4}$}
\authorrunning{ Jamilu, A. et al. }
\institute{Jamilu Adamu  \at
              jamiluadamu88@fugashua.edu.ng 
           \and
             Mehdi Salimi  \at
              mehdi.salimi@kpu.ca
              \and
            Farouq Abba Wabi \at
              farouq.wabi@fuhsa.edu.ng 
               \and
             Jewaidu Rilwan   \at
              jrilwan.mth@buk.edu.ng
             \and   
 \begin{itemize}
 \item[$1$] Department of Mathematics, Federal University Gashua, P.M.B. 1005, Gashua, Yobe State, Nigeria.
 \item[$2$] Mathematics Department, Kwantlen Polytechnics University, BC, Canada.
 \item[$3$] Department of Mathematics, Federal University of Health Sciences, Azare, Bauchi State, Nigeria.
 \item[$4$] Department of Mathematical Sciences, Bayero University Kano, P.M.B. 3011, Kano State, Nigeria.
 \item[$\clubsuit$] Corresponding author.
 \end{itemize}}
 \date{Received: date / Accepted: date}

\maketitle

\begin{abstract}
This paper studies a class of non-cooperative pursuit-evasion differential games in the sequence space $l_{2}.$ The motions of the pursuer and the evader are described by certain first-order differential equations while their control functions are subject to Gronwall-type constraints which impose history-dependent bounds on their control resources. Two fundamental problems are investigated. For the pursuit problem, new strategy  for the pursuer is constructed and sufficient conditions guaranteeing the completion of pursuit are established. For evasion problem, an admissible strategy of the evader ensuring avoidance of capture is obtained, and corresponding sufficient conditions for successful evasion are derived. The proposed approached extends existing results on differential games with classical geometric and integral constraints to a broader class of systems governed by Gronwall-type restrictions.  
\end{abstract}

\newpage
\section{Introduction}
\noindent Differential games represent an important area of modern control theory dealing with dynamical systems involving several decision makers whose objectives are mutually conflicting. The foundations of the theory were established by Isaacs in his pioneering monograph on differential games \cite{Isaacs1965} where the concepts of pursuit, evasion, and optimal strategies were systematically developed. Since then, pursuit-evasion differential games have become one of the most extensively studied classes of conflict-controlled processes because of their numerous applications in military operations, missile guidance, cyber-security, robotics, air traffic control, surveillance systems, and autonomous multi-agent systems. \\

\noindent The theory of pursuit and evasion differential games has undergone substantial development over the last several decades. Significant contributions were made by Petrosyan \cite{Petrosyan 1993}, Lewin \cite{Lewin1994}, Krasovskii and Subbotin \cite{Krasovskii1974,Krasovskii1988,REF15}, Azimov \cite{Azimov 1975}, Chodun \cite{Chodun 1988} who investigated various classes of differential games and established fundamental results concerning solvability conditions, value functions, and strategy constructions. Their works laid the theoretical foundations for the study of pursuit-evasion problems in both finite-dimensional and infinite-dimensional spaces.\\

\noindent A considerable amount of research has been devoted to pursuit-evasion differential games in Hilbert spaces under different classes of control constraints. One of the earliest and most important directions concerns geometric constraints where the admissible controls satisfy point-wise bounds. Another important class involves integral constraints in which the total energy expenditure of each player is bounded. These two classes of constraints have been extensively investigated by many researchers, leading to numerous solvability results and explicit constructions of optimal and admissible strategies. (See for examples \cite{Adamu 2020,Badakaya 2022,A. J. 2022,Badakaya 2024,E. Garcia 2020,Ibragimov 2023,Ibragimov 2025,I. Ibragimov 2025,Kazimirova 2024,Petrov 2025,B. Samatov 2023,Salimi 2019}).\\

\noindent Modern differential game problems in areas such as robotics and medicine incorporate numerous practical considerations into their formulations, necessitating the introduction of additional types of constraints, including mixed constraints \cite{Adamu 2022,Ibragimov 2025,B. Samatov 2023}, Gronwall-type constraints \cite{Ibragimov 2026,J. Rilwan 2023,Rilwan 2023,B. Samatov 2020} and so on. More recently, pursuit-evasion games in infinite-dimensional spaces have been studied under mixed constraints and different types of player dynamics \cite{Adamu 2022,Badakaya 2024}. Badakaya and co-authors \cite{Badakaya 2022,A. J. 2022} considered differential games described by higher-order equations and obtained solvability conditions depending on the distribution of control resources among the players. These investigations substantially enriched the theory of differential games and demonstrated the importance of control constraints in determining the outcome of conflic-controlled processes.\\

\noindent Despite these developments, most of the existing literature is devoted to control functions subject to classical geometric, integral, or mixed constraints. In many practical situations, however, the admissible control resources depend not only on the current state of the system but also on the cumulative history of previous actions. Such situations naturally arise in systems with memory, fatigue effects, resource depletion, adaptive feedback mechanisms, and accumulated energy consumption. Classical control constraints are generally insufficient for accurately describing these phenomena because they neglect the influence of past control actions.\\

\noindent To incorporate memory effects into the mathematical model, it is natural to consider controls satisfying Gronwall-type constraints. Such constraints involve integral inequalities and provide history-dependent bounds on admissible controls. Consequently, they constitute a substantially broader class of restrictions than the standard geometric and integral constraints and allow the modelling of systems whose available resources evolve according to cumulative past behaviour.\\

\noindent Motivated by these considerations, this paper investigates a class of non-cooperative pursuit-evasion differential games in the Hilbert space $l_{2}$, where both players are subject to Gronwall-type constraints. The motions of the pursuer and the evader are governed by specified first-order differential equations. The study focuses on two fundamental problems of differential game theory: the pursuit problem and the evasion problem.\\

\noindent For the pursuit problem, we construct a new admissible strategy for the pursuer and derive sufficient conditions guaranteeing that capture occurs in finite time. For the evasion problem, an admissible strategy of the evader is proposed, and corresponding sufficient conditions ensuring avoidance of capture are established. The obtained strategies explicitly exploit the structure of the Gronwall-type constraints and reveal the influence of memory-dependent control resources on the outcome of the game.\\

\noindent The novelty of the present work lies in the introduction of new strategy constructions for pursuit and evasion under Gronwall-type restrictions and in establishing corresponding solvability conditions in the sequence space $l_{2}$. The obtained results significantly extend several known results on differential games with geometric, integral, and mixed constraints to a more general framework involving history-dependent admissible controls. Therefore, this study contributes to the further development of infinite-dimensional differential game theory and provides new tools for analyzing conflict-controlled systems with memory effects.\\

\noindent The remainder of this paper is organized as follows. Section 2 presents the formulation of the game and some preliminary results. In Section 3, new pursuit strategies are constructed and sufficient conditions for the completion of pursuit are established. Section 4 is devoted to the evasion problem, where admissible evasion strategies and corresponding solvability conditions are derived. Finally, concluding remarks and possible directions for future investigations are given in the last section.

\section{Problem Formulation}
Let the players $p_{j}, \ j \in J=\lbrace 1,2,\dots,m \rbrace $ and $e$ denote $m$ pursuers and one evader, respectively, whose dynamics obey the following system differential equations:
\begin{equation}\label{J1}
\left\{ \begin{array}{ll}
 \frac{dp_{j}}{d\xi} = \lambda(\xi)a_{j}(\xi), & p_{j}(0) = p_{j0}, \\ \\
 
 \frac{de}{d\xi} = \  \lambda(\xi)b(\xi), & e(0) = e_{0},
\end{array} \right.
\end{equation}
where $ p_{j}(\xi), e(\xi), p_{j0}, e_{0}, a_{j}(\xi), b(\xi) \in l_{2}$; $a_{j}(\cdot)$ and $b(\cdot)$ stand for measurable control functions of the players $p_{j}$ and $e$, respectively, and that $\lambda(\cdot)\geq 0$ be scalar measurable function define on the time interval $[0,\theta]$ which denotes the duration of the game.
\begin{definition}
A function $a_{j}(\xi) = (a_{j1}(\xi),a_{j2}(\xi),  \dots)$ whose coordinates $a_{ji}:[0,\theta]\rightarrow R^{1}$ are Borel measurable real-valued functions satisfying the constraints:
\begin{equation}\label{J2}
||a_{j}(\xi)||^{2} \leq \alpha_{j}^{2}+2c\int_{0}^{\xi}\lambda(s)||a_{j}(s)||^{2}ds, \ \ \ \xi \in [ 0,\theta], 
\end{equation}
is called an admissible control of the player $p_{j}$, where $c$ and $\alpha_{j} $ are given positive constants.
\end{definition}
\begin{definition}
A function $b(\xi) = (b_{1}(\xi), b_{2}(\xi),  \dots)$ whose coordinates $b_{j}:[0,\theta]\rightarrow R^{1}$ are Borel measurable real-valued functions that satisfy the constraints:
\begin{equation}\label{J3}
||b(\xi)||^{2} \leq \beta^{2}+2c\int_{0}^{\xi}\lambda(s)||b(s)||^{2}ds, \ \ \ \xi \in [0,\theta],
\end{equation}
is called an admissible control of the player $e$, where $c$ and $\beta$ are positive constants.
\end{definition}
\begin{remark} 
It can be shown that the inequalities (\ref{J2}) and (\ref{J3}) imply, respectively, that
\begin{equation}\label{lem}
\Vert a(\xi)\Vert \leq \alpha_{j} e^{c\int_{0}^{\xi}\lambda(s)ds}, \ \ \Vert b(\xi)\Vert\leq \beta e^{c\int_{0}^{\xi}\lambda(s)ds}.
\end{equation}
Also, note that if $\Vert a_{j}(\xi)\Vert= \alpha_{j} e^{c\int_{0}^{\xi}\lambda(s)ds}$ and $\Vert b(\xi)\Vert= \beta e^{c\int_{0}^{\xi}\lambda(s)ds}, $ then, the inequalities (\ref{J2}) and (\ref{J3}) are respectively satisfied. It is also important to note that (\ref{lem}) does not imply the inequalities (\ref{J2}) and (\ref{J3}).
\end{remark}
Define the sets: 
$$ A_{j}(\alpha)= \left\lbrace  a_{j}=(a_{j1}, a_{j2}, \dots )\in l_{2}:||a_{j}(\xi)||^{2} \leq \alpha_{j}^{2}+2c\int_{0}^{\xi}\lambda(s)||a_{j}(s)||^{2}ds, \ \ \xi \in [ 0,\theta]
  \right\rbrace,$$ and 
$$ B(\beta)= \left\lbrace  b=(b_{1}, b_{2}, \dots )\in l_{2}:||b(\xi)||^{2} \leq \beta^{2}+2c\int_{0}^{\xi}\lambda(s)||b(s)||^{2}ds, \ \ \xi \in [ 0,\theta] \right\rbrace.$$
\begin{remark}
The sets $A_{j}(\alpha)$ and $B(\beta)$ denote all admissible controls of players $p_{j}$ and $e$, respectively. The constraints (\ref{J2}) and (\ref{J3}) are respectively called Gronwall-type constraints. Such constraints involve integral inequalities and provide history-dependent bounds on admissible controls.
\end{remark} 
\noindent Suppose that the players $p_{j}$ and $e$ select admissible controls $a_{j}(\cdot)\in A_{j}(\alpha)$ and $b(\cdot)\in B(\beta)$,  respectively. Then, according to equation (\ref{J1}), the corresponding state variables $ p_{j}(\xi), e(\xi) \in l_{2}$ generate the trajectories of the players as follows:  
\begin{equation}\label{J4}
p_{j}(\xi) = p_{j0} + \int_{0}^{\xi}\lambda(s)a_{j}(s)ds,
\end{equation}
\begin{equation}\label{J42}
e(\xi)  =  e_{0} + \int_{0}^{\xi}\lambda(s)b(s)ds.
\end{equation}
Let
$D(y_{0},\bar{r})=\left\lbrace y \in l_{2}:\Vert y- y_{0} \Vert \leq  \bar{r} \right\rbrace$ and $ S(y_{0},\bar{r})=\left\lbrace y \in l_{2}:\Vert y- y_{0} \Vert =  \bar{r} \right\rbrace $
denote, respectively, the closed ball and the sphere in $l_{2}$ with center $y_{0}$ and radius $\bar{r}$. We now give some useful definitions.
Let $$R_{j}(\theta):=\left( \alpha_{j}^{2}+2c\int_{0}^{\theta}\lambda(s)||a_{j}(s)||^{2}ds\right)^{\frac{1}{2}}\pi(\theta) ,$$ and 
 $$r(\theta):=\left( \beta^{2}+2c\int_{0}^{\theta}\lambda(s)||b(s)||^{2}ds\right)^{\frac{1}{2}}\pi(\theta), $$ where $\pi(\theta)=\int_{0}^{\theta}\lambda(s)ds.$
\begin{definition}
A closed ball $D_{j}\left( p_{j0},R_{j}(\theta)\right) $ is called the attainability domain of the $j^{th}$ player $p_{j}$ from the initial state $p_{j0}$ at the time $\xi=0$ to the terminal time $\theta,$ if the following conditions are satisfied.
\begin{itemize}
\item[(a)] $\Vert p_{j}(\theta)- p_{j0} \Vert \leq R_{j}(\theta).$
\item[(b)] For every $\tilde{p} \in D_{j}\left( p_{j0},R_{j}(\theta)\right) $, there exists an admissible control $a_{j}(\cdot) \in A_{j}(\alpha)$ such that $p_{j}(\theta)=\tilde{p}.$
\end{itemize}
\end{definition}
This is due to the estimate:
\begin{eqnarray}
\Vert p_{j}(\theta)- p_{j0} \Vert &=& \left\Vert \int_{0}^{\theta}\lambda(r)a_{j}(r)dr \right\Vert 
\leq  \int_{0}^{\theta}\lambda(r)\Vert a_{j}(r)\Vert dr  \nonumber\\
& \leq & \int_{0}^{\theta}\lambda(r)\left( \alpha_{j}^{2}+2c\int_{0}^{r}\lambda(s)||a_{j}(s)||^{2}ds \right)^{\frac{1}{2}} dr.  \nonumber
\end{eqnarray}
 Now define $$ F(r)= \left( \alpha_{j}^{2}+2c\int_{0}^{r}\lambda(s)||a_{j}(s)||^{2}ds \right)^{\frac{1}{2}}.$$
 Since
 $$ \int_{0}^{r}\lambda(s)||a_{j}(s)||^{2}ds \leq \int_{0}^{\theta}\lambda(s)||a_{j}(s)||^{2}ds,\ \ \ 0\leq r \leq \theta , $$
 it follows that $F(r)$ is non-decreasing and $F(r)\leq F(\theta).$ Hence,
 \begin{eqnarray}
\Vert p_{j}(\theta)- p_{j0} \Vert &\leq & F(\theta)\int_{0}^{\theta}\lambda(r)dr \nonumber\\
&=& \left( \alpha_{j}^{2}+2c\int_{0}^{\theta}\lambda(s)||a_{j}(s)||^{2}ds \right)^{\frac{1}{2}}\pi(\theta). \nonumber  
\end{eqnarray}
On the other hand, for each $j \in J,$ the player $p_{j}$ can reach any point $\tilde{p} \in D_{j}\left( p_{j0},R_{j}(\theta)\right) $ using admissible control defined by 
\begin{equation}\label{Cont1}
a_{j}(\xi)=\frac{\tilde{p}-p_{j0}}{\pi(\theta)}, \ \ 0\leq \xi \leq \theta.
\end{equation}
Indeed,
\begin{eqnarray}
p_{j}(\theta) &=& p_{j0}+\int_{0}^{\theta}\lambda(\xi)a(\xi)d\xi \nonumber\\
&=& p_{j0}+\int_{0}^{\theta}\lambda(\xi)\frac{\tilde{p}-p_{j0}}{\pi(\theta)}d\xi \nonumber\\
&=& \tilde{p}. \nonumber
\end{eqnarray}
Hence $p_{j}(\theta)=\tilde{p}.$ Moreover, the admissibility of (\ref{Cont1}) follows easily from the fact that $\tilde{p} \in D_{j}\left( p_{j0},R_{j}(\theta)\right)$, that is $\Vert \tilde{p} - p_{j0} \Vert \leq R_{j}(\theta).$ Therefore,
$$
\Vert a_{j}(\xi) \Vert^{2} = \frac{\Vert \tilde{p} - p_{j0} \Vert^{2}}{\pi^{2}(\theta)}
\leq  \alpha_{j}^{2}+2c\int_{0}^{\theta}\lambda(s)||a_{j}(s)||^{2}ds.
$$
\begin{definition}\label{D4}
A closed ball $D\left( e_{0},r(\theta)\right) $ is called the attainability domain of the player $e$ from the initial state $e_{0}$ at the time $\xi=0$ to the terminal time $\theta,$ if the following conditions are satisfied.
\begin{itemize}
\item[(a)] $\Vert e(\theta)- e_{0} \Vert \leq r(\theta).$
\item[(b)] For every $\tilde{e} \in D\left( e_{0},r(\theta)\right) $, there exists an admissible control $b(\cdot) \in B(\beta)$ such that $e(\theta)=\tilde{e}.$
\end{itemize}
\end{definition}
\noindent \textbf{Game Description:} Throughout this paper, the differential game determined by the dynamics (\ref{J1}), together with the control constraints (\ref{J2}) and (\ref{J3}) imposed on the pursuers' controls $a_j(\cdot)$ and the evader's control $b(\cdot)$, respectively, will be denoted by $G^{\ast}$. The main objective of this paper is to investigate the following problems:

\medskip
\noindent \textbf{Problem 1.} Construct an admissible winning strategy for the pursuer and establish sufficient conditions under which the pursuer can guarantee the capture of the evader within the time interval $[0,\theta]$.

\medskip
\noindent \textbf{Problem 2.} Construct an admissible winning strategy for the evader and establish sufficient conditions under which the evader can guarantee successful evasion for all $\xi \geq 0$.

\section{Pursuit Problem }
Here we consider the game $G^{\ast}$ with $J=\{1\}$. The objective of the player $p$ is to guarantee the completion of pursuit by reaching the state of the player $e$, namely, $$p(\xi)=e(\xi),$$ for some time $\xi \in [0,\theta]$. In contrast, the player $e$ aims to prevent this equality from occurring throughout the game. The principal question here is to answer Problem 1 above. To address this problem, we introduce the following subset of the Hilbert space $l_{2}$: $$H=\left\{z\in l_{2}:2\langle e_{0}-p_{0},z\rangle\leq \left(R^{2}(\theta)-r^{2}(\theta)\right)+\left|\left|e_{0}\right|\right|^{2}-\left|\left|p_{0}\right|\right|^{2}\right\}.$$
The set H is a closed half-space in $l_{2}$ and will play a crucial role in establishing sufficient conditions for the solvability of the pursuit problem in the game $G^{\ast}$.\\
 
\noindent \textit{\textbf{Construction of the pursuer's strategy:}}\\
\noindent Let $\theta>0$ be an arbitrary fixed number, and consider the finite time interval $[0,\theta]\subset R$. Suppose that, at the current time $\xi$, the player $p$ has complete information about the initial positions $p_{0}$ and $e_{0}$, as well as the control function $b(\cdot)\in B(\beta)$ employed by the player $e$. We consider a pursuit problem in which the player $p$ adopts a strategy, while the player $e$ employs an arbitrary admissible control $b(\cdot)\in B(\beta)$. We now introduce the strategy of the player $p$.
\begin{definition}
A function $\tilde{a}(\xi,p_{0},e_{0},b(\cdot))$, $\tilde{a}:[0,\theta]\times l_{2}\times l_{2}\times l_{2}\rightarrow l_{2}$, is called a strategy of the player $p$ if the following conditions are satisfied:
\begin{itemize}
\item[(a)] For every $ b(\cdot) \in B(\beta)$, the function $\tilde{a}(\xi,p_{0},e_{0},b(\cdot))$ is Borel measurable for all $\xi \in [0,\theta]$ .
\item[(b)] For every $ b(\cdot) \in B(\beta)$, $\tilde{a}(\xi,p_{0},e_{0},b(\cdot)) \in A(\alpha)$ on the time interval $[0,\theta]$.
\item[(c)] The system  (\ref{J1}) with $a=\tilde{a}(\xi,p_{0},e_{0},b(\cdot))$ admits a unique solution $(p(\cdot),e(\cdot))$. 
\end{itemize}
\end{definition}
Let the strategy of the player $p$ be defined by:
\begin{equation}\label{J5}
a(\xi) =
\frac{e_{0} - p_{0}}{\pi(\theta)} + \int_{0}^{\theta}\frac{\lambda(s)b(s)}{\pi(\theta)}ds, \ \  0 \leq \xi \leq \theta,
\end{equation}  
where $\pi(\theta)=\int_{0}^{\theta}\lambda(s)ds.$
\begin{lemma} \label{J52}
If $e(\theta)\in H$, then the strategy (\ref{J5}) is admissible for every $b(\cdot) \in B(\beta)$ and for all $\xi \in [0,\theta]$. 
\end{lemma}
\begin{proof}
For the strategy (\ref{J5}), we have
\begin{eqnarray} \label{ADM}
||a(\xi)||^{2} &=&\left|\left|\frac{e_{0} - p_{0}}{\pi(\theta)} + \int_{0}^{\theta}\frac{\lambda(s)b(s)}{\pi(\theta)}ds\right|\right|^{2} \nonumber\\
& =& \left|\left|\frac{e_{0} - p_{0}}{\pi(\theta)}\right|\right|^{2}+ \left|\left|\int_{0}^{\theta}\frac{\lambda(s)b(s)}{\pi(\theta)}ds\right|\right|^{2}+2\left\langle\frac{e_{0} - p_{0}}{\pi(\theta)},\int_{0}^{\theta}\frac{\lambda(s)b(s)}{\pi(\theta)}ds\right\rangle  \nonumber\\
&\leq & \frac{||e_{0}-p_{0}||^{2}}{\pi^{2}(\theta)}+ \frac{1}{\pi^{2}(\theta)} 
\left( \int_{0}^{\theta}\lambda(s)\left|\left| b(s)\right|\right|ds\right)^{2} \nonumber\\ 
&+&\frac{2}{\pi^{2}(\theta)}\left\langle e_{0}-p_{0},\int_{0}^{\theta}\lambda(s)b(s)ds\right\rangle .
\end{eqnarray}
Now we estimate second and third term of (\ref{ADM}). We begin with the second term and use the fact that  $b(\cdot) \in B(\beta)$, that is
$$||b(s)|| \leq \left( \beta^{2}+2c\int_{0}^{s}\lambda(r)||b(r)||^{2}dr\right)^{\frac{1}{2}} , \ \ \ s \in [ 0,\theta],$$ 
by this therefore, we have
\begin{eqnarray}\label{TRM2}
\int_{0}^{\theta}\lambda(s)\Vert b(s)\Vert ds  
& \leq & \int_{0}^{\theta}\lambda(s)\left( \beta^{2}+2c\int_{0}^{s}\lambda(r)||b(r)||^{2}dr \right)^{\frac{1}{2}} ds \nonumber\\
& \leq & \int_{0}^{\theta}\lambda(s)\left( \beta^{2}+2c\int_{0}^{\theta}\lambda(r)||b(r)||^{2}dr \right)^{\frac{1}{2}} ds \nonumber\\
&=& \left( \beta^{2}+2c\int_{0}^{\theta}\lambda(r)||b(r)||^{2}dr \right)^{\frac{1}{2}}\int_{0}^{\theta}\lambda(s)ds \nonumber\\
&=& r(\theta).
\end{eqnarray}
This estimates second term of (\ref{ADM}). To estimates third term, we use the fact that $e(\theta) \in H$. That is 
\begin{equation}\label{J53}
2\langle e_{0}-p_{0},e(\theta)\rangle\leq\left(R^{2}(\theta)-r^{2}(\theta)\right)
+||e_{0}||^{2}-||p_{0}||^{2}.
\end{equation}
In view of (\ref{J42}) and (\ref{J53}), we obtain the following estimation
\begin{eqnarray}\label{TRM3}
2\left\langle e_{0}-p_{0},\int_{0}^{\theta}\lambda(s)b(s)ds\right\rangle
 & = & 2\langle e_{0}-p_{0},e(\theta)-e_{0}\rangle \nonumber\\
& = & 2\langle e_{0}-p_{0},e(\theta)\rangle-2\langle e_{0}-p_{0},e_{0}\rangle \nonumber\\
& = & 2\langle e_{0}-p_{0},e(\theta)\rangle+ 2\langle e_{0},p_{0}\rangle-2||e_{0}||^{2} \nonumber\\
& \leq & R^{2}(\theta)-r^{2}(\theta) -||e_{0}-p_{0}||^{2}. 
\end{eqnarray}
We use the inequalities (\ref{TRM2}) and (\ref{TRM3}) in (\ref{ADM}) and then we have
 \begin{eqnarray} 
||a(\xi)||^{2} 
& \leq & \frac{||e_{0}-p_{0}||^{2}}{\pi^{2}(\theta)}+\frac{1}{\pi^{2}(\theta)}r^{2}(\theta)+\frac{1}{\pi^{2}(\theta)}\left(R^{2}(\theta)-r^{2}(\theta)-||e_{0}-p_{0}||^{2}\right) 
\nonumber\\
&=&\frac{1}{\pi^{2}(\theta)}\left( \left(\alpha^{2}+2c\int_{0}^{\theta}\lambda(s)||a(s)||^{2}ds  \right)^{\frac{1}{2}}\pi(\theta)\right) ^{2} \nonumber\\
&=& \alpha^{2}+2c\int_{0}^{\theta}\lambda(s)||a(s)||^{2}ds .\nonumber
\end{eqnarray}
Hence the strategy (\ref{J5}) is admissible.
\end{proof}
\begin{definition}
We say that pursuit can be completed in the game $G^{\ast}$ at the time $\theta$, if there exists a strategy $\tilde{a}(\xi,p_{0},e_{0},b(\cdot))$ of the pursuer such that $p(\theta)=e(\theta)$ for every admissible control $b(\cdot) \in B(\beta)$ of the evader.
\end{definition}
The following theorem provides a sufficient condition for completion of pursuit.
\begin{theorem}
Let $e(\theta)\in H$ in the game $G^{\ast}$. Then the strategy (\ref{J5}) guarantees catching the evader on the time interval $\left[0, \theta \right] $.
\end{theorem}
\begin{proof}
Let the pursuer applies the strategy (\ref{J5}), clearly by lemma (\ref{J52}) the strategy (\ref{J5}) is admissible. Next we show that $p(\theta)=e(\theta)$. 
\begin{eqnarray}
p(\theta) &=& p_{0}+\int_{0}^{\theta}\lambda(\xi)\left(\frac{e_{0} - p_{0}}{\pi(\theta)} + \int_{0}^{\theta}\frac{\lambda(s)b(s)}{\pi(\theta)}ds\right)d\xi   \nonumber\\
& = & p_{0}+\frac{e_{0}-p_{0}}{\pi(\theta)}\int_{0}^{\theta}\lambda(\xi)d\xi
+\int_{0}^{\theta}\int_{0}^{\theta}\lambda(\xi)\frac{\lambda(s)b(s)}{\pi(\theta)}dsd\xi \nonumber\\
& = & e_{0}+\frac{1}{\pi(\theta)}\int_{0}^{\theta}\int_{0}^{\theta}\lambda(\xi)\lambda(s)b(s)dsd\xi \nonumber\\
& = & e_{0}+\frac{1}{\pi(\theta)}\int_{0}^{\theta}\lambda(\xi)d\xi\int_{0}^{\theta}\lambda(s)b(s)ds \nonumber\\
& = & e(\theta) \nonumber
\end{eqnarray}
Hence the proof of the theorem is completed. 
\end{proof}
\section{Evasion Problem}
\begin{definition}
We say that evasion is successful in the game $G^{\ast}$ on the time interval $[0,\infty)$ if there exists an admissible strategy $b(\xi)$ for the evader such that $p(\xi)\neq e(\xi), \qquad \forall, \xi\geq 0,
$ for every admissible control $a(\cdot)\in A(\alpha)$ of the pursuer.
\end{definition}
\begin{definition}
A function $\tilde{b}(\xi,p_{0},e_{0},a(\cdot))$, $\tilde{b}:[0, \infty) \times l_{2} \times l_{2} \times l_{2}\rightarrow l_{2},$ is called the evader's winning strategy if the following conditions hold true.
\begin{itemize}
\item[(a)]  For every $ a(\cdot) \in A(\alpha)$, the function $\tilde{b}(\xi,p_{0},e_{0},a(\cdot))$ is  Borel measurable for all $\xi \geq 0$ .
\item[(b)] For every $ a(\cdot) \in A(\alpha)$, $\tilde{b}(\xi,p_{0},e_{0},a(\cdot)) \in B(\beta)$ for all $\xi \geq 0$ .
\item[(c)] The system  (\ref{J1}) with $b=\tilde{b}(\xi,p_{0},e_{0},a(\cdot))$ has a unique solution $(p(\cdot),e(\cdot))$. 
\end{itemize}
\end{definition}
\noindent \textit{\textbf{Construction of the evader's strategy:}}\\
\noindent Let $p_{0}\neq e_{0}, \ \varrho=\frac{p_0-e_0}{||p_0-e_0||}$, $\pi(\xi)=\int_{0}^{\xi}\lambda(s)ds$ and $\tau=\beta^{2}-\alpha^{2}\geq 0$. It is not difficult to show that
\begin{equation}\label{E11}
e^{2c\pi(\xi)}=1+2c\int_{0}^{\xi}\lambda(s)e^{2c\pi(s)}ds.
\end{equation}
Defined a resolving function as follows:  
\begin{equation}\label{P11}
\Lambda(\xi)=\langle a(\xi),\varrho\rangle + \sqrt{\tau e^{2c\pi(\xi)}+\langle a(\xi),\varrho\rangle^{2}},
\end{equation}
where $\langle a(\xi),\varrho\rangle$ is the scalar product of the vectors $a(\xi)$ and $\varrho$ in $l_{2}$. We can easily see that the function (\ref{P11}) is well-defined, continuous and non-negative for all $a(\cdot)\in A(\alpha)$. With this and using the condition $\lambda(\xi)\geq 0$ for all $\xi\geq 0$, we have
\begin{equation} \label{P13}
\int_{0}^{\xi}\lambda(s)\Lambda(s)ds \geq 0.
\end{equation}
 We define strategy of the evader as follows:  
\begin{equation}\label{PS2}
b(\xi) = \left\{\begin{array}{ll}
a(\xi)-\Lambda(\xi)\varrho, & 0 \le \xi \le \theta\\
0,   &  \xi > \theta  
\end{array}\right.
\end{equation}  
The following theorem address \textbf{problem 2} above.   
\begin{theorem}\label{TRM1}
Let $e_{0}\neq p_{0}$, then the evader's strategy (\ref{PS2}) ensures 
$p(\xi)\neq e(\xi)$ for all $\xi \geq 0$ in the game $G^{\ast}$.
\end{theorem}
\begin{proof}
Firstly, we show that the strategy (\ref{PS2}) is admissible for every $a(\xi) \in A(\alpha)$, $\xi \geq 0$. Indeed
\begin{eqnarray}\label{PS3}
||b(\xi)||^{2} &=&\left|\left|a(\xi)-\langle a(\xi),\varrho\rangle \varrho - \varrho\sqrt{\tau e^{2c\pi(\xi)} +\langle a(\xi),\varrho\rangle^{2}}\right|\right|^{2} \nonumber\\
&=&\left|\left|a(\xi)-\langle a(\xi),\varrho\rangle \varrho \right|\right|^{2}+\tau e^{2c\pi(\xi)}+\langle a(\xi),\varrho\rangle^{2} \nonumber\\
&-& 2\left\langle a(\xi)-\langle a(\xi),\varrho\rangle \varrho , \varrho\sqrt{\tau e^{2c\pi(\xi)} +\langle a(\xi),\varrho\rangle^{2}}\right\rangle \nonumber\\
&=& \left|\left|a(\xi)\right|\right|^{2}-\langle a(\xi),\varrho\rangle^{2}
+\tau e^{2c\pi(\xi)}+\langle a(\xi),\varrho\rangle^{2} \nonumber\\
&=& \left|\left|a(\xi)\right|\right|^{2}+\tau e^{2c\pi(\xi)}.
\end{eqnarray}
Since $a(\xi) \in A(\alpha)$, then using (\ref{E11}), equation (\ref{PS3}) becomes 
\begin{eqnarray}
||b(\xi)||^{2}&\leq & \alpha^{2}+2c\int_{0}^{\xi}\lambda(s)\Vert a(s)\Vert^{2}ds +\tau e^{2c\pi(\xi)} \nonumber\\
& = & \alpha^{2}+2c\int_{0}^{\xi}\lambda(s)\Vert a(s)\Vert^{2}ds 
+\tau\left(1+2c\int_{0}^{\xi}\lambda(s)e^{2c\pi(\xi)} ds\right) \nonumber\\
& =& \alpha^{2}+\tau+2c\int_{0}^{\xi}\lambda(s)\left( \left|\left|a(s)\right|\right|^{2}+\tau e^{2c\pi(\xi)}\right) ds \nonumber\\
& = & \beta^{2} + 2c\int_{0}^{\xi}\lambda(s)\Vert b(s)\Vert^{2}ds
\end{eqnarray}
Hence, strategy (\ref{PS2}) is admissible. Next, we show that $p(\xi)\neq e(\xi),$ for all  $\xi \geq  0$. This is as follows:  
\begin{eqnarray}\label{JJ6}
p(\xi)-e(\xi)&=& p_{0}-e_{0}+\int_{0}^{\xi}\lambda(s)a(s)ds-\int_{0}^{\xi}\lambda(s)b(s)ds 
\nonumber\\
&=&p_{0}-e_{0} + \int_{0}^{\xi}\lambda(s)\Lambda(s)\varrho ds \nonumber\\
&=& \varrho \mu(\xi),
\end{eqnarray}
where
$$
\mu(\xi):=\Vert p_{0}-e_{0}\Vert+\int_{0}^{\xi}\lambda(s)\Lambda(s)ds.
$$ 
Clearly, $\mu(0)=\Vert p_{0}-e_{0}\Vert > 0$, then the conclusion of the Theorem (\ref{TRM1}) can be established if we show that $\mu(\xi)>0$ for all $\xi>0$. This follows from  (\ref{P13}) and hence the proof is completed. 
\end{proof}

\section{Conclusion}
\noindent This paper considered a class of non-cooperative pursuit-evasion differential games in the sequence space $l_{2}$, where the players' motions are described by first-order differential equations and their controls are constrained by Gronwall-type conditions. These constraints generalize the classical geometric and integral restrictions by incorporating history-dependent bounds on the control resources of the players.\\

\noindent For the pursuit problem, a new admissible strategy for the pursuer was constructed, and sufficient conditions guaranteeing the completion of pursuit were established. It was shown that whenever $e(\theta)\in H$ in the game $G^{\ast}$, then the strategy (\ref{J5}) guarantees catching the evader within a finite time. For the evasion problem, an admissible strategy of the evader was proposed, and sufficient conditions ensuring the impossibility of capture were derived. Under these conditions, the evader is able to avoid interception for all admissible actions of the pursuer.\\

\noindent Therefore, the obtained results provide a characterization of the solvability of both pursuit and evasion problems under Gronwall-type constraints and extend several known results in the literature concerning differential games with geometric, integral, and mixed constraints. The developed approach demonstrates that memory-dependent restrictions can be effectively incorporated into the analysis of differential games in infinite-dimensional spaces.\\

\noindent Future investigations may focus on the study of multi-player pursuit-evasion games, higher-order and fractional dynamical systems, stochastic differential games, and the determination of optimal strategies and the value of the game under Gronwall-type constraints.
 
\section*{Availability of data and materials}
No data are associated with this manuscript.

\section*{Conflict of interest}
The authors declare that they have no competing interest.



\end{document}